\documentclass[12pt]{amsart}
\usepackage{amsthm,amsmath,amssymb}
\usepackage{xypic}

\newcommand{\ep}{\varepsilon}
\newcommand{\CC}{{\mathbb{C}}}

\newcommand{\HH}{{\mathbb{H}}}
\newcommand{\FF}{{\mathbb{F}}}

\newcommand{\PP}{{\mathbb{P}}}

\newcommand{\ZZ}{{\mathbb{Z}}}

\newcommand{\calA}{{\mathcal A}}

\newcommand{\calM}{{\mathcal M}}

\newcommand{\calX}{{\mathcal X}}

\def\ZZ{{\mathbb Z}}

\def\PP{{\textbf P}}

\newcommand{\op}{\operatorname}

\newcommand{\SpF}{\op{Sp}(2g,\FF_2)}
\newcommand{\SpZ}{\op{Sp}(2g,\ZZ)}
\newcommand{\SpFthree}{\op{Sp}(6,\FF_2)}
\newcommand{\SpFtwo}{\op{Sp}(4,\FF_2)}
\newcommand{\SpZthree}{\op{Sp}(6,\ZZ)}
\newcommand{\SpZtwo}{\op{Sp}(4,\ZZ)}

\newcommand{\Pic}{\op{Pic}}

\newcommand{\T}{\Theta}
\renewcommand{\t}{\theta}

\renewcommand\tt[2]{\t\left[\begin{matrix}#1\\ #2\end{matrix}\right]}
\newcommand\stt[2]{\t\left[\begin{smallmatrix}#1\\ #2\end{smallmatrix}\right]}
\newcommand\tc[2]{{\left[\begin{matrix}#1\\ #2\end{matrix}\right]}}
\newcommand\stc[2]{{\left[\begin{smallmatrix}#1\\ #2\end{smallmatrix}\right]}}
\def\tch#1#2{{\left[\begin{matrix}#1\\ #2\end{matrix}\right]}}
\def\tt#1#2{{\t\tch{#1}{#2}}}
\newcommand{\s}{\sigma}

\def\aa{\overline{\mathcal{A}}}

\theoremstyle{plain}
\newtheorem{thm}{Theorem}%[section]
\newtheorem{lm}[thm] {Lemma}
\newtheorem{prop}[thm]{Proposition}

\theoremstyle{definition}

\newtheorem{rem}[thm]{Remark}

\begin{document}
\title{On the Coble quartic }

\author{Samuel Grushevsky}
\address{Mathematics Department, Stony Brook University,
Stony Brook, NY 11790-3651, USA}
\email{sam@math.sunysb.edu}
\thanks{Research of the first author is supported in part by National Science Foundation under the grant DMS-12-01369.}
\author{Riccardo Salvati Manni}
\address{Universit\`a ``La Sapienza'', Dipartimento di Matematica, Piazzale A. Moro 2, I-00185, Roma,   Italy}
\email{salvati@mat.uniroma1.it}

\begin{abstract}
We review and extend the known constructions relating Kummer threefolds, G\"opel systems, theta constants and their derivatives, and the GIT quotient for 7 points in $\PP^2$ to obtain an explicit expression for the Coble quartic.
The Coble quartic was recently determined completely in \cite{sturmfels}, where (Theorem 7.1a) it was computed completely explicitly, as a polynomial with 372060 monomials of bidegree $(28,4)$ in theta constants of the second order and theta functions of the second order, respectively. Our expression is in terms of products of theta constants with characteristics corresponding to G\"opel systems, and is a polynomial with 134 terms. %=1*8+7*2*4+7*5*2=8+56+70=134

Our approach is based on the beautiful geometry studied by Coble \cite{coblebook} and further investigated by Dolgachev and Ortland in \cite{doorbook}, and highlights the geometry and combinatorics of syzygetic octets of characteristics, and the corresponding representations of $\SpFthree$. One new ingredient is the relationship of G\"opel systems and Jacobian determinants of theta functions.

In genus 2, we similarly obtain a short explicit equation for the universal Kummer surface, and relate modular forms of level two to binary invariants of  six points on $\PP^1$.
\end{abstract}
\maketitle

\section{Introduction}
The existence of this paper is due to the appearance of the paper \cite{sturmfels}, and to the interest of its authors Qingchun Ren, Steven Sam, Gus Schrader, and Bernd Sturmfels in this circle of classical ideas in algebraic geometry. Their work inspired us to revisit and reconsider the classical constructions originating with Coble.
In fact    they study the defining ideal of the universal Kummer threefold in $\PP^7\times\PP^7$ (see \cite[Conj.~8.6]{sturmfels} for a complete conjectural
description of this ideal), and give an explicit equation for the (universal) Coble quartic: the Jacobi modular form of weight (16,4) or equivalently the bidegree $(28,4)$ polynomial on $\PP^7\times\PP^7$ in theta constants and  theta functions of the second order.
The Coble quartic is the unique polynomial invariant under the action of the symplectic group, and its eight partial derivatives with respect to the second set of variables give the 8 defining cubic equations for Kummer threefolds. The investigation of this hypersurface goes back to Coble himself. In his book \cite[p.~106]{coblebook} Coble gives  an implicit equation for the quartic, writing it as
$$\alpha Q_1+2\alpha_1 Q_2+\dots2\alpha_7 Q_8+4\alpha_{423} Q_9+\dots+4\alpha_{456} Q_{15},$$
where $Q_i$  are explicit quartics in theta functions of the second order, listed explicitly in section \ref{sec:coble} of this paper.
About the coefficients $\alpha_{\ldots}$ Coble states \cite[p.~196]{coblebook}:

{\em ``The 15 coefficients  $\alpha$ of the quartic spread  $L^4$ in $S_7$
can be expressed linearly with numerical coefficients in terms of the G\"opel invariant of $\PP^2_{7}$
and conversely.\dots''}

The aim of this paper is to follow  Coble's idea and  give an equation for Coble's quartic in this way,
with explicit formulas for the coefficients $\alpha$. One new tool in our approach compared to \cite{sturmfels} are the gradients at $z=0$ of theta functions with odd characteristics. It is a classical result that these are related to bitangents of plane quartic curves, see \cite{case1} and \cite{case2} (and \cite{doorbook} for much more of the beautiful classical geometry). This classical relationship gives a geometric viewpoint of our paper, providing a connection (in fact constructing a homomorphism) between the ring of GIT invariants of  seven points in $\PP^2$  and the  ring of modular forms of genus 3 and level 2. This allows us to express the ``G\"opel invariants'' (or Fano G\"opel coordinates in \cite{sturmfels})
in terms of theta constants with characteristics.  We then use representation theory
of $\SpFthree$ to construct the Coble quartic.  Indeed,  the image of the 15-dimensional  vector  space spanned by the coefficients $\alpha$ of the Coble quartic is a 15-dimensional space of  modular forms of genus 3 and level 2, of weight 14, known to be related to Gop\"el systems, and for
which we thus obtain expressions in terms of G\"opel systems. Similarly the invariant quartics in theta functions of the second order
form another 15-dimensional irreducible representation of the group $\SpFthree$. Since the tensor product of these two representations contains a unique copy of a trivial representation, this trivial representation is generated by the Coble quartic, and symmetrizing under the action of $\SpFthree$ must produce it (up to a constant factor) --- this is argued in theorem
\ref{cobleformula}, where an expression for the Coble quartic as a symmetrization (actually, under a smaller group, of
order 135) is obtained. We then perform a straightforward computation in Maple\textsuperscript{\textregistered},
giving $\alpha_{\ldots}$ above explicitly in terms of a basis for the space of modular forms generated by Fano configurations --- the resulting
expression is an explicit polynomial with 134 monomials, presented in theorem \ref{explicitcoble}. Of course
the resulting expression agrees with the result of \cite{sturmfels}, though it is written in terms of different
variables (theta constants with characteristics).

We note also that the story for $g=2$ in many aspects parallels the situation for $g=3$, but is of course easier. We develop this story
in parallel with the Coble quartic, and in theorem \ref{thm:genus2} determine explicitly the relation between the ring of modular forms of genus 2 and level 2, and the GIT quotient of 6 points on $\PP^1$ (similar to the relationship of modular forms of genus 3 and level 3, and the GIT quotient of 7 points on $\PP^2$ discussed above).

\section*{Acknowledgements}
The first author thanks Qingchun Ren, Steven V Sam, Gus Schrader, and especially Bernd Sturmfels for interesting conversations.  We are grateful to Steven Sam and Bernd Sturmfels for many useful comments on an earlier version of this manuscript. We thank Qingchun Ren for  verifying numerically that our results agree with those of \cite{sturmfels} and finding a computational mistake in an earlier version of our Maple\textsuperscript{\textregistered} spreadsheet.

The second author thanks Corrado De Concini for   conversations about the ring of invariants.

We thank Eberhard Freitag for computing for  us the tensor product of representations of $\SpFthree$ used in theorem \ref{cobleformula}.

\section{Combinatorics of theta characteristics}
In this section we recall some known facts about the action of the symplectic group $\SpF$ on the set of theta characteristics. The main references are the classical books on theta functions, eg.~Wirtinger \cite{wirtingerbook}, Krazer \cite{krazer}, Coble \cite{coblebook}, with a more modern treatment given by Igusa \cite{igusabook}, and many details that we need are investigated in \cite{smlevel2}.  Almost all of these facts are of course discussed in \cite{sturmfels}, but we collect here all that we need,  as a potential convenient reference for theta constants, and use this to fix notation.

A {\em theta characteristic} $m$ is an element
of $\FF_2^{2g}$ (where we think of the elements of $\FF_2$ as being 0 and 1), which we will often write as $m=\tc{m'}{m''}$  with $m'$ and $m''$ considered as row vectors in $\FF_2^g$. We define the parity as
\begin{equation}\label{eg:sign}
e(m):=(-1)^{^{t}m'\cdot m''}
\end{equation}
and say that $m$ is even or odd according to whether $e(m)$ is equal to $1$ or $-1$, respectively.

For any triple $m_1, m_2, m_3$ of characteristics we set
\begin{equation}\label{eg:syzygy}
e(m_1, m_2, m_3):=e(m_1)\cdot e(m_2)\cdot e(m_3)\cdot e(m_1+m_2+m_3),
\end{equation}
and call a triple {\em syzygetic} or {\em azygetic} depending on whether this number is $1$ or $-1$, respectively.
%We call a sequence $m_1,\dots,m_r$ of characteristics  essentially independent   if they satisfy no linear relations with an even number of terms (since we are working over $\FF_2$, this is equivalent to saying that any sum of an even number of distinct characteristics in the set is not equal to zero).

The action of $\SpF$ on the set of characteristics is given by
$$
  \gamma\cdot m:=\left(\begin{array}{cc}
   D &  -C \\
   -B  & A \end{array}\right)\tc{m'}{m''}+ \left(\begin{array}{c}
  diag(C^t D)\\
  diag(A^t B)\end{array}\right),
$$
where as usual we write $\gamma\in\SpF$ in the block form as $\gamma=\left(\begin{matrix} A&B\\ C&D\end{matrix}\right)$. This affine-linear action
preserves the parity, azygy/syzygy, and linear relations with an even number of terms. That is to say, the orbits of this action on tuples of characteristics are described as follows: there exists an element of $\SpF$ mapping a sequence of characteristics $m_1,\ldots,m_l$ to a sequence of characteristics $n_1,\ldots,n_l$ if and only if for any $1\le i,j,k\le l$ we have $e(m_i)=e(n_i),$ $e(m_i,m_j,m_k)=e(n_i,n_j,n_k)$, and linear relations among $m_i$'s with an even number of terms are in bijection with such relations for the $n_i$'s, see \cite{smlevel2}.

A {\em fundamental system} of characteristics is a set of
$2g+2$ characteristics such that any triple is azygetic; we will consider fundamental systems as {\it unordered} sets of theta characteristics. By the above, two fundamental systems belong to the same $\SpF$ orbit if and only if they contain the same number of odd characteristics. In fact, a fundamental system with $k$ odd theta characteristics exists if and only if $k$ is congruent to $g$ modulo 4, see \cite{fayriemann}.
A {\em special fundamental system} is a fundamental system containing $g$ odd characteristics and $g+2$ even characteristics --- all special fundamental systems form one $\SpF$ orbit.
%Moreover, all fundamental systems form one orbit under the so-called theta group, an extension of $\SpF$ by translations,
% see \cite{fayriemann}.

\smallskip
In this paper, we are mostly interested in the cases $g=2$ and $g=3$.   For genus 2, fundamental systems consist of 6 characteristics, of which either 2 or 6 may be odd. The unique fundamental system with 6 odd characteristics is simply the set of all 6 odd characteristics:
\begin{equation}\label{example2}
\tc{01}{01},\tc{01}{11},\tc{10}{10},\tc{10}{11},\tc{11 }{01 }, \tc{11 }{10 }.
\end{equation}
In genus 2, we also have the following easy combinatorial
\begin{lm}\label{lm:genus2}
For any azygetic triple of odd characteristics $m_1,m_2,m_3$ there exist a unique even characteristic
$n_0$ such that the quadruple $n_0,m_1, m_2,m_3$ is azygetic, and it is given by $n_0=m_1+m_2+m_3$.
\end{lm}

 The rest of this section is devoted to the case $g=3$. Then fundamental systems consist of 8 characteristics, of which either 7 or 3 may be odd. In the first case, we have
fundamental systems of the form
$$
  n_0, m_1, m_2,\dots, m_7
$$
with $n_0$ even  and all $m_i$ odd. Classically, in this case the set of odd characteristics
$m_1, m_2,\dots, m_7$ is  called  an Aronhold set. We refer to \cite{doorbook} for a beautiful modern exposition of the theory and many classical and new results.  Thus the  group $\SpFthree$ acts transitively and faithfully on the collection of Aronhold sets, and the number of  ordered different Aronhold sets is equal to $|\SpFthree|=36\cdot8!=1451520$.

An example of an Aronhold set is the following:
\begin{equation}\label{example}
\tc{111}{111},\tc{110}{100},\tc{101}{001},\tc{100}{110},\tc{010}{011},\tc{001}{101},\tc{011}{010}
\end{equation}
This Aronhold set is completed to a fundamental system by adding the zero characteristic
$\tc{000}{000}$.

We now collect some needed results about the combinatorics of Aronhold sets for $g=3$:
\begin{lm}\label{lm1}
For any azygetic triple of odd characteristics $m_1,m_2,m_3$ there exist 6 even characteristics
$n_1,\ldots,n_6$ such that the quadruple $m_1, m_2,m_3, n_j$ is azygetic. One of these characteristics (label it $n_6$) is equal to the sum $n_6=m_1+m_2+m_3$. Then the set
$$
  m_1, m_2,m_3, n_1,\dots, n_5
$$
forms a  special fundamental system. In particular, every azygetic triple of odd characteristics is contained in (and defines) a unique, up to permutations, special fundamental system.
\end{lm}
\begin{lm}\label{lm2}
If $m_1,m_2,m_3$ and $m_1, m_4, m_5$ are two azygetic triples of odd characteristics (sharing one characteristic) that are subsets of a fundamental system $n_0, m_1, m_2,\dots, m_7$, then the intersection of the special fundamental systems defined by $m_1,m_2,m_3$, and by $m_1,m_4,m_5$, respectively, consists of $m_1$ and $n_0$.
\end{lm}

Since the group $\SpFthree$ acts transitively on azygetic triples of odd characteristics, it is enough to verify these lemmas for just one such odd triple $m_1,m_2,m_3$. This can be checked by hand, but more easily follows from the following combinatorial
\begin{lm}\label{lm3}
For any fundamental system $n_0,m_1,\ldots,m_7$, the other 21 odd characteristics are obtained as all
possible sums $n_0+m_i+ m_j $  with $1\leq i<j\leq 7$, while the 36 even characteristics are $n_0$ together with the 35 sums $m_i+ m_j +m_k$  with $1\leq i<j<k\leq 7$.
\end{lm}
\begin{rem}
Explicitly, if an azygetic odd triple $ m_1, m_2,m_3$ is part of the fundamental system   $n_0, m_1, m_2,\dots, m_7$, then the corresponding special fundamental system is given by
$$
  \lbrace m_1,m_2,m_3,n_0,m_4+m_5+m_6,m_4+m_5+m_7,m_4+m_6+m_7,m_5+m_6+m_7 \rbrace.
$$
\end{rem}

We define the symplectic form on $\FF_2^6$ by
$$
  e(m, n):=(-1)^{m'^tn''-m''^tn'}.
$$
Then the number of Lagrangian subspaces of $\FF_2^6$, called G\"opel systems, is equal to 135, see  \cite{smrelations}.
We call a G\"opel system a {\it Fano configuration} if all 8 characteristics in it are even, and call it a {\it Pascal configuration} if four characteristics in the Lagrangian space are even, and four are odd. In genus 3 there are 30 Fano configurations and 105 Pascal configurations.
In both cases, there exists a unique affine subspace of $\FF_2^6$ modeled on the G\"opel systems, which consists only of even characteristics --- for a Fano configuration, this is the G\"opel system itself.

All Fano configurations can be obtained as follows: fix an Aronhold set $m_1,\ldots,m_7$. Then for any set of 7 triples of indices among $1,\ldots, 7$, such that any two triples have exactly one element in common, the seven corresponding triple sums $m_i+m_j+m_k$ form the set of non-zero elements of a Fano configuration, and all Fano configurations are obtained this way. For example, we could choose the 7 triples as
\begin{equation}\label{F}
  F:=\left\{ (123), (145), (167), (247),(256), (346), (357)\right\},
\end{equation}
so that for the Aronhold set example given in (\ref{example}) the corresponding Fano configuration formed by zero and the above 7 elements is
\begin{equation}\label{ex1}
\tc{000}{000},\tc{100}{010},\tc{001}{010},\tc{101}{000},\tc{001}{000},\tc{101}{010},\tc{000}{010},\tc{100}{000}.
\end{equation}

Similarly, starting with a fixed Aronhold set, the non-zero elements of any Pascal configuration can be described by taking sums of elements of the following form (so that there are three triple sums all sharing the same one element, each giving an even characteristic lying in the Pascal configuration, and the other 4 odd elements of a Pascal configuration are obtained as double sums forgetting the common characteristic, and the common characteristic itself)
\begin{equation}\label{P}
  P:=\left\{ (123), (145), (167), (1),(23), (45), (67)\right\},
\end{equation}
so that for the Aronhold set given in (\ref{example}) we get the Pascal configuration
\begin{equation}\label{ex2}
\tc{000}{000},\tc{100}{010},\tc{001}{010},\tc{101}{000},\tc{111}{111},\tc{011}{101},\tc{110}{101},\tc{010}{111}.
\end{equation}

\section{ Modular forms}
In this section we review the notation on modular forms and theta constants.
We denote by $\HH_g$ the {\em Siegel upper half-space} --- the space of
complex symmetric $g\times g$ matrices with positive definite imaginary
part. An element $\tau\in\HH_g$ is called a {\em period matrix}, and
defines the complex abelian variety $X_{\tau}:=\CC^g/\ZZ^g+\tau \ZZ^g$. The
group $\Gamma_g:={\rm Sp}(2g,\ZZ)$ acts on $\HH_g$ by automorphisms: for
$$
  \gamma=\left(\begin{array}{cc}
  A&B\\ C&D\\
  \end{array}\right)
  \in{\rm Sp}(2g,\ZZ)
$$
the action is
$\gamma\circ\tau:=(A\tau+B)(C\tau+D)^{-1}$. The quotient of $\HH_g$ by the
action of the symplectic group is the moduli space of principally
polarized abelian varieties (ppav): $\calA _g:=\HH_g/{\rm Sp}(2g,\ZZ)$.
We define the {\em level subgroups}  of the symplectic group to be
$$
  \Gamma_g(n):=\left\lbrace\gamma
  \in\SpZ\, |\, \gamma\equiv\left(\begin{array}{cc}
   1&0\\ 0&1\\
  \end{array}\right)  {\rm mod}\, n \right\rbrace
$$
The corresponding {\em level  $n$ moduli space of ppav} is denoted
$\calA_g(n):=\HH_g/\Gamma_g(n)$.

A function $F:\HH_g\to\CC$ is called a {\em modular form of weight $k$ and multiplier  $\chi$
with respect to a subgroup $\Gamma\subset\Gamma_g$} if
$$
  F(\gamma\circ\tau)=\chi(\gamma)\det(C\tau+D)^kF(\tau),\quad \forall \gamma
  \in\Gamma,\ \forall \tau\in\HH_g.
$$
  We shall write  $[\Gamma, k,\chi]$ for this  space. We omit the character if it is trivial.
We shall define the ring of modular forms as
\begin{equation}\label{modforms}
  A(\Gamma)=\bigoplus_{k=0}^{\infty}[\Gamma, k].
\end{equation}
This is a finitely generated graded ring.
For any theta characteristic $m\in\FF_2^g$, we define the {\em theta function with
characteristic $m$} to be the map $\theta_m:\HH_g\times\CC^g\to\CC$ given by
$$
  \t_m(\tau, z):=\sum_{p\in\ZZ^g} \exp \pi i \left[
  \left(p+\frac{m'}{2}\right)^t\cdot\left(\tau(p+\frac{m'}{2})+2(z+
  \frac{m''}{2})\right)\right].
$$
Sometimes we will write $\tt{m'}{m''}(\tau,z)$ for $\t_m(\tau,z)$.

For $\ep\in\FF_2^g$ we also define the {\em second order theta function with characteristic $\ep$} to be
$$
\T[\ep](\tau,z):=\tt{\ep}{0}(2\tau,2z).
$$
The transformation law for theta functions under the action of the
symplectic group is given in \cite{igusabook}:
$$
\theta\left[\gamma  \left(\begin{matrix}m'\\ m''\end{matrix}\right)
   \right](\gamma\circ\tau,(C\tau+D)^{-t})z)= $$
 $$\phi(m,\gamma)\det
 (C\tau+D)^{1/2}e^{\pi i \left( z^t(C\tau+D)^{-1}Cz\right)}\tt {m'} {m''}(\tau,z),
$$
where $\phi$ is some complicated explicit function   (which we will discuss in more detail before theorem \ref{explicitcoble}), and the action of $\gamma$ on characteristics is taken modulo $2$. It is further known (see \cite{igusabook}, \cite{smnullwerte}) that $\phi|_{z=0} $ does not depend on $\tau$. Thus the values of theta functions   at $z=0$, called {\em theta constants}, are modular forms of weight one half with multiplier, with respect to $\Gamma_g(2)$. Similarly it is known that the theta constants of second order are modular forms of weight one
half with respect to a certain normal subgroup $\Gamma_g(2,4)\subset\SpZ$ (containing $\Gamma_g(4)$ and contained in $\Gamma_g(2)$).

All odd theta constants with characteristics vanish identically, as the corresponding theta functions are odd functions of $z$, and thus there are $2^{g-1}(2^g+1)$ non-trivial theta constants with characteristics, corresponding to even theta characteristics $m$. All theta functions of the second order are even with respect to $z$, and all $2^g$ theta constants of the second order are not identically zero.

Given a set of $g$ odd characteristics $m_1,\ldots, m_g$, one constructs the {\em Jacobian determinant} from the gradients with respect to $z$ of
the corresponding theta functions, evaluated at $z=0$:
$$
  D(m_1,\dots, m_g)(\tau):=\left.\vec{\rm grad}_z\theta_{m_1}(\tau,z)\wedge\dots\wedge \vec{\rm grad}_z\theta_{m_g}(\tau, z)\right|_{z=0},
$$
which is a modular form of weight $\frac{g+2}{2}$ with respect to $\Gamma_g(2)$, with a suitable multiplier.

  In genus 1, the famous Jacobi derivative formula
$$
 D\left(\tc11\right)=-\pi\tt00\tt01\tt10
$$
expresses the only Jacobian determinant as a product of theta constants. This formula has been generalized to
genus 2 and 3 by Igusa \cite{igusajacobi}, and the results are as follows.
In genus 2, for any special  fundamental system   $m_1,m_2, n_1, n_2, n_3, n_4$ we have
\begin{equation}\label{Jacder2}
  D(m_1, m_2) = -\pi^2\t_{n_1}\t_{n_2}\t_{n_3}\t_{n_4}
\end{equation}

In genus 3, for any azygetic triple of odd characteristics $m_1,m_2,m_3$ forming a special fundamental system with $n_1,\ldots,n_5$ (see lemma \ref{lm1}), we have
\begin{equation}\label{Jacder}
  D(m_1, m_2,m_3) = -\pi^3\t_{n_1}\t_{n_2}\t_{n_3}\t_{n_4}\t_{n_5}.
\end{equation}
In general, for higher genus it is known that the Jacobian determinant is not a polynomial in theta constants, and various results were obtained by Igusa \cite{igusajacobi} and Fay \cite{fayriemann}. Moreover, a Jacobian determinant associated to
a non-azygetic set of characteristics is not a polynomial in the theta constants whenever $g\geq 3$.

\medskip
In the remaining sections of the text, we will use the generalized Jacobi derivative formulas for genera 2 and 3 to relate modular forms, configurations spaces of points, and the Coble quartic. We first deal with the easier case of genus 2 in the following two sections.

\section{Abelian surfaces with level 2 structure, and the GIT quotient of $(\PP^1)^6$}
The moduli space of abelian surfaces with a level two structure, $\calA_2(2)=\HH_2/\Gamma_2(2)$
admits various compactifications. Indeed, the space of modular forms $[\Gamma_2(2),2]$ for $\Gamma_2(2)$
of weight $2$ is 5-dimensional, spanned by the fourth powers of theta constants with characteristics $\theta_n^4$.
Choosing an appropriate set of characteristics $n_1,\ldots,n_5$ such that $X_0:=\theta_{n_1}^4(\tau),\ldots,X_4:=\theta_{n_5}^4$
are linearly independent defines an embedding $\calA_2(2)\hookrightarrow\PP^4$. The closure
of the image --- the Satake compactification $\calA_2(2)^{Sat}$ --- is given by one equation, Igusa quartic, see \cite{igugenus22}:
\begin{equation}\label{eq:igusa}
\begin{aligned}
I(X_0,\ldots,X_4):&=(X_0X_1+X_0X_2+X_1X_2-X_3X_4)^2\\ &-4X_0X_1X_2(X_0+X_1+X_2+X_3+X_4).
\end{aligned}
\end{equation}
This result can be restated by saying that the ring $A(\Gamma_2(2))$ of modular forms of genus 2, level 2, and even
weight is generated by the fourth powers of 5 theta constants with characteristics, with the only relation being
the Igusa quartic:
$$
 A(\Gamma_2(2))=\bigoplus_{k=0}^{\infty} [\Gamma_2[2], 2k]=\CC[X_0,\dots, X_4]/I(X_0,\dots, X_4).
$$

Alternatively, note that any indecomposable abelian surface is a hyperelliptic Jacobian, and thus
is determined by the 6 branch points of the hyperelliptic cover of $\PP^1$ (up to automorphisms of $\PP^1$); the level two
structure corresponds to an ordering of these 6 points, and thus an alternative birational model of $\calA_2(2)$
is obtained as the GIT quotient of $(\PP^1)^6$ under the action of $\op{PSL}(2,\CC)$. We now recall this
construction.

Denote $\calX(6)\subset (\PP^1)^6$ the subset where all the points are distinct; the action of $\operatorname{PSL(2,\CC)}$ on
$\PP^1$ extends to its action (diagonally) on $\calX(6)$. The configuration
space is then defined as
$$
 X(6)^{o}=\calX(6)/\operatorname{PSL(2,\CC)}.
$$
In non-homogeneous coordinates $x_1,\ldots,x_6$ on $\CC^6\subset(\PP^1)^6$, degree $k$   forms on $X(6)^o$,
which we call {\em binary invariants of degree $k$}, are given by polynomials $P\in \CC[x_1,\dots, x_6]$ such that
$$P(\gamma\cdot x)=\prod_{i=1}^6 (cx_i+d)^{-k} \cdot P(x)$$
where we denote
$$(\gamma\cdot x)_i= (ax_i+b)(cx_i+d)^{-1}$$
for every
$$\gamma=\begin{pmatrix} a&b\\c&d\end{pmatrix}\in \operatorname{PSL(2,\CC)}.$$
We  denote by $S(6)_{k}$ the space  generated by such forms, and let $S(6):=\bigoplus_{k=0}^{\infty}S(6)_{k}$
be the graded  ring  of binary invariants.

Recall that a {\em tableau} is a way of filling a matrix $\begin{pmatrix}i_1 &i_2 &i_3\\ j_1 &j_2 &j_3\end{pmatrix}$ by the numbers from 1 to 6
in such a way that we have
$$i_1<i_2<i_3,\qquad i_1<j_1,\quad i_2<j_2,\quad i_3<j_3.$$
A tableau is called {\em standard} if moreover it satisfies
$j_1<j_2<j_3$.

There are 5 standard tableaux enumerated as follows:
$$
\begin{pmatrix}1&3&5\\2&4&6\end{pmatrix},\ \begin{pmatrix}1&2&5\\3&4&6\end{pmatrix},
\ \begin{pmatrix}1&3&4\\2&5&6\end{pmatrix},\ \begin{pmatrix}1&2&4\\3&5&6\end{pmatrix},
\ \begin{pmatrix}1&2&3\\4&5&6\end{pmatrix}.
$$
To any tableau $N$ we associate the invariant
$$
B(N):=(x_{i_1}-x_{j_1})(x_{i_2}-x_{j_2})(x_{i_3}-x_{j_3})\in S(6)_1 .$$
The space spanned by all these polynomials has dimension 5, and a basis is given by polynomials associated to
the standard tableaux, which we denote $T_0(x),\ldots,T_4(x)$ in the above ordering. Kempe \cite{kempe} proved that $T_i(x)$ generate the
ring of invariants $S(6)$; the invariants $T_i$ thus define a smooth embedding
$$
  T:X(6)^{o}\to \PP^{4},
$$
and we denote by $X(6)$ the closure of the image $T(X(6)^{o})$. This is called the GIT quotient $X(6)=(\PP^1)^6//\operatorname{PSL(2,\CC)}$,
and gives an alternative compactification of the moduli space of indecomposable abelian surfaces with a level two structure.
It turns out that the ideal of relations among binary invariants is generated by the Segre cubic polynomial
\begin{equation}\label{eq:segre}
 S(T_0,\dots,T_4):=T_1T_2T_4-T_3(T_0T_4+T_1T_2-T_0T_1-T_0T_2+T_0^2)
\end{equation}
(see \cite{doorbook}), and hence
$$
  X(6)=\operatorname{Proj}\left(\CC[T_0,\dots,T_{4}]/S (T_0,\dots, T_4)\right).
$$
\begin{rem}
It is in fact known that the Igusa quartic and Segre cubic are dual hypersurfaces, see \cite{doorbook}.
Moreover, if one embeds a Kummer surface as a quartic surface in $\PP^3$ by using theta functions of the
second order, its equation is of the form
\begin{equation}\label{kumsur}
K(\tau,z)=\alpha_0(x_{00}^4+x_{01}^4+x_{10}^4+x_{11}^4)+ 2\alpha_1(x_{00}^2x_{10}^2+x_{01}^2x_{11}^2 )
\end{equation}
$$+2\alpha_2(x_{00}^2x_{01}^2+x_{10}^2x_{11}^2 )+2\alpha_3(x_{00}^2x_{11}^2+x_{10}^2x_{01}^2 )+
4\alpha_4(x_{00}x_{01}x_{10}x_{11} ),
$$
where $x_\ep=\Theta[\ep](\tau,z)$. It then turns out that the coefficients $\alpha_i$ satisfy the Segre cubic equation and certain inequalities, see \cite{doorbook}.

The above equation for the universal Kummer surface can be obtained  as the Fourier-Jacobi expansion
of the unique degree 16 polynomial (the Schottky-Igusa form) in theta constants of the second order for genus 3 that vanishes identically, whose image under the Siegel $\Phi$ operator gives exactly the Igusa quartic. In the following section we compute the coefficients $\alpha_i$ explicitly using a different approach --- which then generalizes to the case of genus 3.
\end{rem}

\section{Modular forms of genus 2 and level 2, and binary invariants of six points on $\PP^1$}
We now investigate the relation between the two birational models of $\calA_2(2)$ constructed
in the previous section. That is to say, we will investigate the relation between modular forms of genus 2 and level 2, and binary invariants of 6 points on $\PP^1$, constructing a pair of dual homomorphisms from one graded algebra to the other. One of these --- the homomorphism associating a binary invariant to a modular forms was classically studied by Thomae \cite{thomae} in the nineteenth century, with further work done by Igusa \cite{igusaproj} and
the second author \cite{smslopetheta}.
Indeed, working in arbitrary genus, Thomae defined a homomorphism from the ring $A(\Gamma_g(2))$ of modular forms of genus $g$
and level 2 to the ring $S(2g+2)$ of binary invariants of points on $\PP^1$. This isomorphism is obtained essentially by
restricting the modular forms to the locus of hyperelliptic Jacobians. The corresponding expression for theta constants
of hyperelliptic Jacobians is known as Thomae's formulas, and the results in genus 2 are as follows:
\begin{thm}[Thomae \cite{thomae}, see also Igusa \cite{igusaproj}]
For $g=2$ for any even theta characteristic $n$ write
$$n=m_{i_1}+m_{i_2}+m_{i_3}=m_{i_4}+m_{i_5}+m_{i_6}$$
where $i_1,\ldots,i_6$ is a suitable ordering of the odd characteristics. Let
\begin{equation}\label{psitheta}
  \psi^*(\t_n^4):=\pm (x_{i_1}-x_{i_2 })(x_{i_2}-x_{i_3})(x_{i_3}-x_{i_1})(x_{i_4}-x_{i_5 })(x_{i_5}-x_{i_6})(x_{i_6}-x_{i_4})
\end{equation}
and extend $\psi$ to a homomorphism of the ring of modular forms of genus 2 and level 2, generated by $\t_n^4$. Then the map $\psi$ defines an injective degree preserving homomorphism
$$
 %\psi^*:A(\Gamma_2(2)=\CC[X_0,\dots, X_4]/I(X_0,\dots, X_4)\to S_6=\CC[T_0,\dots,T_{4}]/S (T_0,%\dots,T_4)
\psi^*:A(\Gamma_2(2))\to S_6.
$$
\end{thm}
For degree reasons the associated rational map
$$\psi:X(6)\dashrightarrow \calA_2(2)^{Sat}$$
interpreted  as a map from the Segre cubic to the Igusa quartic is a map given via quadrics, so it is just the  map  given by the  gradients of the cubic equation (once we choose a suitable  basis).

We will now construct a homomorphism of graded algebras
 $$
%\phi^*: \CC[T_0,\dots,T_{4}]/S (T_0,\dots, T_4)\to \CC[X_0,\dots, X_4]/I(X_0,\dots, X_4)
\phi^*: S_6\to A(\Gamma_2(2))
$$
such
that the associated rational map
$$\phi:\calA_2(2)^{Sat}\dashrightarrow X(6)$$
will be the inverse   of  the  rational map $\psi$.

To construct the homomorphism $\phi^*$, we will use the Jacobian determinants discussed above. Indeed,
for any binary  invariant written as a polynomial $P=P(x_i-x_j)$ in variables $x_i-x_j$ for $1\leq i<j\leq6$, we set
$$\phi^*(P):= P(D(m_i,m_j)),$$
where the resulting expression is then written as a scalar modular form by utilizing
the generalized Jacobi derivative formula.
\begin{prop}
With the above notation we have for any even theta characteristic $n$
$$\phi^*(\psi^*)(\t_n^4)=\chi_5^2\t_n^4$$
with
$$\chi_5=\prod_{n\, {\rm even}} \t_n.$$
Vice versa,
$$\psi^*(\phi^*(x_{i_1}-x_{j_1})(x_{i_2}-x_{j_2})(x_{i_3}-x_{j_3}))=\Delta^{1/2}\cdot (x_{i_1}-x_{j_1})(x_{i_2}-x_{j_2})(x_{i_3}-x_{j_3})$$
with $$\Delta^{1/2}=\prod_{1\leq i<j\leq 6}(x_i-x_j).$$
\end{prop}
\begin{proof}
Using equation (\ref{psitheta}) defining the map $\psi$, the computations with Jacobian determinants done by Fiorentino
\cite{fiorentino}, and lemma \ref{lm:genus2}, we get
$$
  D(m_{i_1},m_{i_2})D(m_{i_1},m_{i_3})D(m_{i_2},m_{i_3})
$$
$$
  =\pm D(m_{i_4},m_{i_5})D(m_{i_4},m_{i_6})D(m_{i_5},m_{i_6})
$$
$$=\mp\pi^6 \chi_5\t_m^2$$
Multiplying the first two lines of the formula thus gives $\pi^{12}\chi_5^2\t_m^4$ proving the formula for $\phi^*\circ\psi^*$.

Vice versa, we compute
$$\phi^*(x_{i_1}-x_{j_1})(x_{i_2}-x_{j_2})(x_{i_3}-x_{j_3})=$$
$$\t_{m_{i_1}+m_{j_1} +m_{j_3} }^2  \t_{ m_{i_1}+m_{j_1} +m_{i_3} }^2\t_{m_{i_1}+m_{j_1} +m_{j_2}}^2\cdot$$ $$\t_{m_{i_1}+m_{j_1} +m_{i_2}}^2 \t_{m_{j_1}+m_{i_2} +m_{j_2}}^2\t_{ m_{i_1}+m_{i_2} +m_{j_2}}^2$$
Applying $\psi$ and using the computations in \cite{fiorentino}, we finally get
the formula for $\psi^*\circ\phi^*$.
\end{proof}
\begin{rem} We observe  that
$$\psi ( \chi_5)=c\Delta^{1/2},$$
and the zero set of each form is the complement of the locus of indecomposable abelian surfaces (Jacobians of smooth
hyperelliptic curves) with level 2 structure, in both varieties.
\end{rem}
As a consequence of the above discussion we have
\begin{thm} \label{thm:genus2}
The maps $\phi$ and $\psi$ restricted to the locus of indecomposable ppav within $\calA_2(2)$, i.e.~to the locus of Jacobians of smooth hyperelliptic genus two curves, are the inverses of each other.
Hence choosing  suitable  bases in the space of  binary invariants of degree one  and in the space of  modular forms of weight induces the dual  maps between the Igusa quartic and the Segre cubic.
\end{thm}
As an application, we compute the coefficients $\alpha_i$ of the universal Kummer surface (\ref{kumsur}), proving along the way that they are modular forms of weight six for $\Gamma_2(2)$. To this end, we apply the homomorphism $\phi^*$ to these coefficients of the universal Kummer surface. A very similar strategy applies in the more difficult case of genus 3, yielding the equation of the Coble quartic. Here we only sketch the argument for genus 2, referring to the genus 3 case studied below for a more detailed discussion.

It turns out that the five coefficients $\alpha_0,\ldots,\alpha_4$ of the universal Kummer surfaces, and the five invariant quartic polynomials in $x_{ij}$ in (\ref{kumsur}) each span a five-dimensional representation of $\SpFtwo$. Let us denote these representations by $W$ and $W'$, respectively; it turns out that these representations are isomorphic, and that the tensor product $W\otimes W'$ contains a unique copy of the trivial representation of $\SpFtwo$. We define a suitable subgroup $\Gamma_{2,0}(2)$ of $\SpZtwo$ of index 15 containing $\Gamma_2(2)$ (see (\ref{gamma30}) below for the analogous definition in genus 3) such that each $W$ and $W'$ contain a unique $\Gamma_{2,0}(2)$-invariant vector, which we denote $v$ and $v'$, and then the symmetrization of $v\otimes v'$ under $\SpZtwo/\Gamma_{2,0}(2)$ must give the element of $W\otimes W'$ that corresponds to the trivial representation of $\SpFtwo$, which is then equal to the universal Kummer surface. This situation is completely analogous to the case of the Coble quartic, discussed in much more detail below, and the resulting expression for the universal Kummer surface is
$$K(\tau,z) =s_1(x_{00}^4+x_{01}^4+x_{10}^4+x_{11}^4)- 2(s_1+2s_2)(x_{00}^2x_{10}^2+x_{01}^2x_{11}^2 )$$
\begin{equation}\label{kumsurexpl}-2(s_1+2s_3)(x_{00}^2x_{01}^2+x_{10}^2x_{11}^2 )
-2(s_1+2s_4)(x_{00}^2x_{11}^2+x_{10}^2x_{01}^2 )
\end{equation}
$$+8(s_1+s_2+s_3+s_4+2s_5)(x_{00}x_{01}x_{10}x_{11} ),
$$
where
$$
s_1=\left(\stt{00}{00}\stt{00}{01}\stt{00}{10}\stt{00}{11}\right)^{-2}\chi_{5}^2,
$$
$$
s_2=\left(\stt{00}{00}\stt{00}{01}\stt{10}{00}\stt{10}{01}\right)^{-2}\chi_{5}^2,
$$
$$
s_3=\left(\stt{00}{00}\stt{00}{10}\stt{01}{00}\stt{01}{10}\right)^{-2}\chi_{5}^2,
$$
$$
s_4=\left(\stt{00}{00}\stt{00}{11}\stt{11}{00}\stt{11}{11}\right)^{-2}\chi_{5}^2,
$$
$$
s_5=\left(\stt{00}{00}\stt{01}{00}\stt{10}{00}\stt{11}{00}\right)^{-2}\chi_{5}^2,
$$
and we recall that $x_\ep=\Theta[\ep](\tau,z)$.

This completes our results for the case of genus 2. The rest of the paper will be devoted to a
detailed exposition of the case of genus 3, culminating in formulas for the Coble quartic.

\section{The configuration space of seven points on $\PP^2$}
Similar to the above relation of modular forms of genus 2, and invariants of points on $\PP^1$,
the G\"opel systems discussed above are related to configurations of points on $\PP^2$, and we now recall this construction, from \cite{coblebook}, \cite{doorbook}, \cite{kondo}.

Let $\PP_2^7$ be the GIT quotient of $(\PP^2)^{\times 7}$ under the diagonal action of $PGL(3,\CC)$, and let $R_2^7$ be its ring of invariants, see \cite{doorbook} for the construction and details. Recall that $R_2^7$ is finitely generated: to any G\"opel system one can associate (see \cite{doorbook} and \cite{coblebook}) a $15$-dimensional subspace $V\subset R_2^7$ of invariants of degree $3$ as follows.

Let $p_i, p_j, p_k\in\PP^2$, and  $v_i, v_j, v_k\in\CC^3\setminus\{0\}$ be such that $v_i\mapsto p_i$ under the projection. Denote then $(ijk):=v_i\wedge v_j\wedge v_k$, and associate to the Fano configuration given by (\ref{F}) the  function
\begin{equation}\label{GL}
  G_F:= (123)(145)(167)(247)(256)(346)(357),
\end{equation}
and associate to Pascal configuration given by (\ref{P}) the function
\begin{equation}\label{GM}
  G_P:=(123)(145)(167)\Big((246)(356)(257)(347)-(256)(357)(247)(346)\Big).
\end{equation}
It turns out that the 135 functions $G_F$ and $G_P$ span a 15-dimensional space of degree 3   invariants $V$, a basis of which is given by
$G_{F_1},\ldots G_{F_{15}}$ corresponding to suitable 15 Fano configurations. The linear relations among the $G_F$ and the $G_P$ are related to two-dimensional isotropic subspaces of $\FF_2^6$,  see \cite{coblebook},  \cite{doorbook}, or \cite{kondo}.

This  $15$-dimensional space $V$ is an irreducible representation of $\SpFthree$. Moreover, choosing a basis in it given by suitable 15 $G_F$'s defines a $\SpFthree$-equivariant rational map $\PP_2^7\dasharrow\PP^{14}$, birational onto the  image (the same map is defined also from a six-dimensional  ball quotient, see \cite{kondo} for details). The image of this map, the so-called G\"opel variety, is described completely in \cite{sturmfels}. Moreover, according to \cite[p.~196]{coblebook}, this  space appears in the definition of the coefficients of the Coble quartic.  In \cite{sturmfels} the authors determine a $15$-dimensional space of modular forms of weight $14$ (i.e.~polynomials of degree 28 in theta constants with characteristics) corresponding to this  space. We will do the same, by using  and interpreting the constructions and results originating with Coble, in a more geometric way.

\section{Del Pezzo surfaces of degree 2 and plane quartics}
We recall from \cite{coblebook} or \cite{doorbook} that the double cover of $\PP^2$ branched along a smooth plane quartic is a degree 2 del Pezzo surface. Conversely, the anticanonical model of any degree 2 del Pezzo surface $S$ is a double cover of $\PP^2$ branched along a smooth quartic.

A del Pezzo surface of degree 2 is obtained by blowing up seven points $p_1,\dots, p_7\in \PP^2$ in general position. The Picard lattice of $S$  is isomorphic to $(1)\oplus(-1)^{\oplus7}$, and it is  generated
by $h_0$ (the total transform of the hyperplane section)  and  the exceptional curves $h_1, \dots, h_7$. The orthogonal complement of the anti-canonical class $3h_0-h_1-\ldots-h_7$ in $\Pic_\ZZ(S)$ then turns out to be isomorphic to the lattice $E_7$. The surface $S$ contains 56 exceptional  $-1$ curves, which split into $28$ pairs, such that the curves in each pair are interchanged by the deck transformation of the cover $S\to \PP^2$. For any set of seven disjoint $-1$ curves, contracting them defines a morphism $\pi: S\to \PP^2$, called a {\em geometric marking} of the degree 2 del Pezzo surface. The  number of contractions
is equal to the order of the Weyl group $W(E_7)$.
Moreover, there is an exact sequence
$$
 1\to (\pm1)\to W(E_7)\to\SpFthree\to 1,
$$
where the kernel of the map to $\SpFthree$ is the deck transformation of the cover.
Hence, a geometric marking of a degree 2 del Pezzo surface corresponds to a level  $2$ structure on the Jacobian of the plane quartic $C$. Under the anticanonical map the chosen seven pairs of disjoint $-1$ curves map to seven bitangents of $C$ that form an {\em Aronhold set of bitangents}, that is to a $7$-tuple $\{\ell_1,\dots, \ell_7\}$  of bitangents such that for each triple $\ell_i,\ell_j,\ell_k$ the corresponding six points of tangency with $C$ do not lie on a conic.

Recall that bitangents of a plane quartic are in bijection with odd theta characteristics in genus 3 (the two points of tangency give an effective square root of the canonical bundle), and Aronhold sets of characteristics thus correspond to Aronhold sets of bitangents. We thus obtain a birational map between  $\PP_2^7$ and $\calA_3(2)$ (the moduli space of genus 3 curves with a level 2 structure), defined away from the hyperelliptic locus in $\calM_3$, see \cite{case1},\cite{lehavi}.

\section{Relations among modular forms in genus 3}
We denote by $\aa_3$  and $\aa_3(2)$  the Satake compactifications  of $\calA_3$  and  $\calA_3(2)$, respectively. Let
$$
  \chi_{18}(\tau):=\prod_{m\ {\rm even}}\t_m(\tau, 0)
$$
be the product of all even theta constants with characteristics --- it is a modular form of weight $18$ for the entire group $\SpZthree$, as each of the 36 theta constants is a modular form of weight $1/2$. Recall also that hyperelliptic genus 3 curves are characterized by having one vanishing theta constant, and thus the equation
$\chi_{18}=0$ defines in $\aa_3$ the closure of the hyperelliptic locus. Recall also that any non-hyperelliptic genus 3 curve is a plane quartic, and thus the complement in $\aa_3$ of the zero locus $\lbrace\chi_{18}=0\rbrace$ is the moduli space of plane quartics. We denote this space by  $\aa_3^0$, and similarly denote $\aa_3(2)^0$ the level 2 cover of the moduli of non-hyperelliptic genus 3 curves.

From the discussion in the previous section, it follows that $\aa_3(2)^0$ is isomorphic to each of the following: the moduli space of plane quartics together with an Aronhold set of  bitangents;  the  moduli space of marked degree 2 del Pezzo surfaces; the open subset of the GIT quotient   $\PP_2^7$  where the points  are in general position, i.e.~no three lie on a line, and no six lie on a conic.

Using these identifications  we associate to the  functions $G_F$ defined by (\ref{GL})   suitable modular forms (with trivial  character) for the group  $\Gamma_3(2)$. Indeed, for a fixed Aronhold set $m_1,\dots, m_7$ and a Fano configuration $F$ given as triples of characteristics $M_i=(i_1, i_2, i_3)$ with $1\le i_1\le i_2\le i_3\le 7$, we set
$$
  H(F)(\tau):=\frac{ D(M_1)\dots D(M_7)}{\t_{n_0}^7}(\tau),
$$
where we denote
$$
  D(M_i)(\tau):= D(m_{i_1}, m_{i_2}, m_{i_3}) (\tau).
$$

By Jacobi's derivative formula (\ref{Jacder}), we can express each $D(M_i)$ as a product of 5 theta constants (thus a modular form of weight $5/2$), and by the discussion in lemma \ref{lm2} it follows that each of these products contains $\theta_{n_0}$, so that each $D_{M_i}$ is divisible by $\theta_{n_0}$, and thus $H(F)(\tau)$ is indeed a (holomorphic)  modular  form of weight $14=7\cdot (5/2-1/2)$. Similarly we can define the modular form $H(P)(\tau)$ corresponding to any Pascal configuration $P$ as in (\ref{GM}). Summarizing, as a consequence of the previous identifications we have

\begin{thm}
For  any $\tau\in\aa_3(2)^0$    let  $X_{\tau}:=(p_1, \dots, p_7)\in\PP_2^7$ be the corresponding 7-tuple of points. Fix an Aronhold  set $m_1, \dots, m_7$; then for every Fano configuration $F$, we have
$$
  G_{F}(X_\tau)=\t_{n_0}^7(\tau)H(F)(\tau).
$$
Moreover, we  have
$$
  H(F)(\tau)=\frac{\chi_{18}(\tau)} { \prod_{n\in F}\t_n(\tau)},
$$
and thus all $H(F)(\tau)$ are modular forms with respect to $\Gamma_3(2)$ (with trivial character).
\end{thm}
\begin{proof}
Indeed, on $\HH_3$ by the discussion above we can express $H(F)$ as a monomial in theta constants by using Jacobi's derivative formula, yielding the expression above. The  fact that the modular form  $H(F)(\tau)$ has trivial character is a consequence of a result of Igusa \cite{igusaring2}.
\end{proof}

\begin{rem}
  The construction above associating modular forms to Fano configurations depends on the choice of the Aronhold  set. For a different Aronhold set, the correspondence between the $G_{F}$ and the monomials $H(F)$ would be different, but the denominator would still be of the form $\t_{n}^7$, and thus on the set of plane quartics, which is the complement of the locus $\lbrace \chi_{18}(\tau)=0\rbrace$, the projective map given by the set of all $H(F)(\tau)$ would be the same, since no theta constant vanishes.
\end{rem}
We know that the 15-dimensional space $V$ is spanned  by the functions $G_F$. Let us then denote by $W$ the 15-dimensional space spanned by the   functions  $H(F)$.  The  linear relations among the $G_F$ and the $G_P$ induce linear relations among  the $H(F)$ and $H(P)$. The consistency of these relations follows from Riemann's quartic addition theorem for theta constants. Indeed, recall (see for example \cite{rafabook} or \cite{vgvdg})   that in genus three Riemann's quartic addition theorem for theta constants with characteristics has the form
\begin{equation}\label{Riemann}
  r_1=r_2+r_3,
\end{equation}
where each $r_i$ is a product of four theta constants with characteristics forming an even coset of a two-dimensional isotropic space. Hence the characteristics appearing in the product $r_ir_j$, $i\neq j$  are even cosets of Lagrangian spaces (i.e.~affine subspaces consisting only of even characteristics, modeled on G\"opel configurations as vector subspaces), and we thus get
\begin{equation}\label{Riemchi}
  \frac{\chi_{18}}{r_1r_2}+\frac{\chi_{18}}{r_1r_3}- \frac{\chi_{18}}{r_3r_2}= \frac{\chi_{18}}{r_1r_2r_3} (r_3+r_2-r_1)=0.
\end{equation}
The group
\begin{equation}\label{gamma30}
  \Gamma_{3,0}(2):=\left\lbrace\gamma
  \in\SpZthree\, |\, C\equiv 0 \mod 2\right\rbrace
\end{equation}
has index 135 in $\SpZthree$, using Maple\textsuperscript{\textregistered} (or it can also be seen by using Igusa's going down  process),  we found that the subspace of the $15$-dimensional space $W$ invariant under the  action of $\Gamma_{3,0}(2)$
is one-dimensional, and is thus spanned by the manifestly $\Gamma_{3,0}(2)$-invariant modular form
\begin{equation}\label{HF1}
 H(F_1)(\tau):=\frac{\chi_{18}}{\prod_{m''}\tt{0}{m''}(\tau)}
\end{equation}
corresponding to the Fano configuration $F_1:=\left\lbrace \tc{0}{m''}\right\rbrace$ consisting of all characteristics with top vector zero

We conclude this section by remarking that all the functions $H(F)$ and $H(P)$ induce the rational map
$$
  \phi:\aa_3(2)\dasharrow\PP^{14}=\PP(W).
$$
Thus one can easily see \cite[Thm.~7.1 (b)]{sturmfels}:
\begin{prop}
The base locus of $\phi$ consists of reducible points in $\aa_3(2)$.
\end{prop}
\begin{proof}
We can choose a basis of $W$ consisting of monomials in the theta constants. Since for a Jacobian of a smooth hyperelliptic genus 3 curve, precisely one theta constant vanishes, the map $\phi$ is still well-defined on the locus of hyperelliptic Jacobians. However, for a reducible point at least 6 theta constants with azygetic characteristics vanish (this condition characterizes the reducible locus, see  \cite{frsmhp}). The complementary set of characteristics in each  monomial  defining the map $\phi$ is a G\"opel system, i.e.~a syzygetic octet of characteristics. A G\"opel system cannot contain an azygetic triple of characteristics, and thus all monomials defining the map $\phi$ vanish identically on the reducible locus, so that the map $\phi$ is undefined there.
\end{proof}

We observe that  the  map is defined exactly along $\calM_3(2)$. In particular, the map $\phi$ is well-defined on the hyperelliptic locus.

\section{The Coble quartic}
\label{sec:coble}
We combine the results and constructions summarized above to obtain a formula for the Coble quartic as a suitable symmetrization.

By the work of Coble \cite[p.~196]{coblebook}, the Coble quartic is the linear combination of 15 terms that  span an irreducible representation of $\SpFthree$, which we now describe. We denote theta functions of the second order by
\begin{equation}\label{x}
x_{\ep}:=\T[\ep](\tau,z)=\tt{\ep}{0}(2\tau,2z)
\end{equation}
for $\ep\in\FF_2^g$; these are coordinates for the projective space that is the target of the Kummer map. Let then $W'$ be the $15$-dimensional  vector space spanned by quartics $Q_1,\ldots,Q_{15}$ in $x_\ep$ invariant under translations, enumerated explicitly in the next section, following Coble's notation in \cite{coblebook} (and also given in \cite{sturmfels}). Here we only note that
$$Q_1:=x_{000}^4 +x_{001}^4 +x_{010}^4 +x_{100}^4 +x_{110}^4 +x_{101}^4 +x_{011}^4 +x_{111}^4$$
spans the unique $\Gamma_{3,0}(2)/\Gamma_3(2)$ invariant in $W'$. This $15$-dimensional space $W'$ forms an irreducible representation of the group $\SpFthree$, where $\SpFthree$ acts on theta functions of the second order (explicit formulas are given in the next section). It can be verified that the representation $W'$ is in fact isomorphic to $W$, but we keep the notation distinct. Finally, we set
$R(\tau, z):=H(F_1)(\tau)Q_1(\tau, z)$ to be the product of the two such invariants, and obtain a formula for
the Coble quartic.
\begin{thm}\label{cobleformula}
Up to a constant factor, the Coble quartic is equal to
$$
 \sum_{\gamma\in\SpZthree/\Gamma_{3,0}(2)} \gamma(R(\tau, z))=$$
$$\sum_{\gamma\in\SpZthree/\Gamma_{3,0}(2)}
\det(C\tau+D)^{-16} e^{8\pi i  z^t(C\tau+D)^{-1}z} R(\gamma\tau,(C\tau+D)^{-t}z).$$
\end{thm}
\begin{proof}
Using Magma\textsuperscript{\textregistered}, Eberhard Freitag  determined the decomposition of $W\otimes W'$ into irreducible representations under the action of $\SpFthree$. It turns out that the dimensions of the irreducible summands are equal to $1,35,84,105$,
each occurring with multiplicity one, and since the Coble quartic is invariant under $\SpFthree$, it must generate the 1-dimensional trivial subrepresentation of $W\otimes W'$.  (From Schur's lemma it follows a priori that there is at most a single copy of the trivial representation in this tensor product, but the computation above actually guarantees the existence of this invariant, which thus proves the existence of the Coble quartic!)

Recall that we have $\Gamma_3(2)\subset\Gamma_{3,0}(2)\subset\SpZthree$, so we have $$\Gamma_{3,0}(2)/\Gamma_3(2)\subset\SpFthree=\SpZthree/\Gamma_3(2).$$
Since the modular form $H(F_1)$ defined in (\ref{HF1}) spans the $\Gamma_{3,0}(2)/\Gamma_3(2)$ invariant line in $W$,
and similarly $Q_1$ spans the $\Gamma_{3,0}(2)/\Gamma_3(2)$ invariant line in $W'$, the expression
$$v_0= \sum_{\gamma\in\SpFthree/\Gamma_{3,0}(2)} \gamma(v\otimes v')\in W\otimes W'$$
is invariant under $\SpFthree$, and thus, unless it is zero, generates the one-dimensional irreducible
summand of $W\otimes W'$. Indeed,
to prove that  $v_0$ is not identically zero, we
recall that in fact $W$ and $W'$ are isomorphic representations, this isomorphism must send $H(F_1)$ to $Q_1$, and
it is enough to check that the sum above is non-zero under such an identification. To this end, we compute using
the transformation formula for theta constants given in \cite{smthetanullw}
$$
  \sum_{\gamma\in\SpFthree} \gamma(H(F_1)^2)= \vert\Gamma_{3,0}(2)/\Gamma_3(2)\vert\sum_{i=1}^{135} H(G_i)^2,
$$
where $G_1,\ldots,G_{135}$ is some enumeration of all the G\"opel systems.
This form is not identically  zero, in fact the first Fourier coefficient  (for some  lexicographic ordering)
of each $H(G_i)^2$ is positive, hence $v_0\in W\otimes W$ cannot be 0, and must be proportional to the Coble quartic.
\end{proof}

\section{An explicit formula for the Coble quartic}
Using a computer, we compute explicitly the expression for the Coble quartic given by Theorem \ref{cobleformula}, by determining the orbit of $H(F_1)Q_1$ under the group $\SpFthree/\Gamma_{3,0}(2)$. We recall that  the symplectic group acts on the theta functions with characteristics and of the second order by the ``slash'' action according to the following rules
$$\gamma(\tt\alpha\beta(\tau, z)):=
 \det(C\tau+D)^{-1/2} e^{\pi i  z^t(C\tau+D)^{-1}z} \tt\alpha\beta(\gamma\tau,(C\tau+D)^{-t}z).$$
$$\gamma(\T[\ep](\tau, z)):=
 \det(C\tau+D)^{-1/2} e^{2\pi i  z^t(C\tau+D)^{-1}z} \T[\ep](\gamma\tau,(C\tau+D)^{-t}z),$$
and this action restricts to an action on the corresponding theta constants. In working with this action, a recurrent difficulty is the multiplier $\kappa$ that appears in the transformation formula for theta function. However, this multiplier depends only on $\gamma$, and not on the characteristics, and since it is an eighth root of unity, it disappears when working with polynomials in theta constants of degree divisible by 8 --- which is always our situation.

Since the conventions in the literature vary, and there are various typos in signs,
for easy reference here are the relevant formulas from \cite{igusabook},\cite{runge1},\cite{smlevel2}
(note especially that in the first  formula the theta function on the right-hand-side is written with characteristic in $\FF_2^6$, which accounts for the extra sign $(-1)^{\lfloor\frac{\beta+S\alpha -\mathop{diag} S}2\rfloor}$).
Omitting the multiplier, we then have the following transformation formulas
$$\left(\begin{matrix} 1&S\\ 0&1\end{matrix}\right)\tt\alpha\beta=\left(\frac{1+i}{\sqrt2}\right)^{\alpha^t (-S\alpha-2\mathop{diag} S+4\lfloor\frac{\beta+S\alpha +\mathop{diag} S}2\rfloor)}\tt\alpha{\beta+S\alpha +\mathop{diag} S},$$
$$\left(\begin{matrix} 0&1\\ -1&0\end{matrix}\right)\tt\alpha\beta=\left(\frac{1+i}{\sqrt2}\right)^{-2\alpha^t
\beta}\tt\beta\alpha,$$
$$\left(\begin{matrix} 1&S\\ 0&1\end{matrix}\right)\T[\ep]=i^{\ep^t S\ep}\T[\ep],$$
$$\left(\begin{matrix} 0&1\\ -1&0\end{matrix}\right)\T[\ep]=\left(\frac{1}{\sqrt2}\right)^g\sum_\alpha (-1)^{\alpha^t\ep}\T[\alpha].$$

To compute the Coble quartic, we then enumerate the 30 Fano configurations and the 105  Pascal configurations. For each Pascal configuration
$P$ there exists a unique pair of Fano configurations $F',F''$ such that there exist three
syzygetic quartets of characteristics $S_1,S_2,S_3$ satisfying
$$
 P=S_2\sqcup S_3,\quad F'=S_1\sqcup S_2,\quad F''=S_1\sqcup S_3.
$$
As described above, Riemann's quartic addition theorem then yields for any Pascal configuration $P$ the expression
$H(P)=\pm H(F')\pm H(F'')$, where the signs are given explicitly by Igusa
\cite{igusachristoffel}. Once this is done, in the same manner one investigates Riemann's
quartic addition theorem given by 3 syzygetic quartets where
each $S_i\sqcup S_j$ is a Pascal configuration --- substituting the expressions for $H(P)$
then yields various relations among the modular forms $H(F)$. Solving these allows one to find
an explicit basis for the space $W$. Explicitly, such a basis can be chosen to be given by
the following 15 Fano configurations:
$$F_1=\left\lbrace\stc{000}{000}, \stc{000}{001},\stc{000}{010}, \stc{000}{011}, \stc{000}{100},\stc{000}{101}, \stc{000}{110}, \stc{000}{111}\right\rbrace,$$
$$F_2=\left\lbrace\stc{000}{000}, \stc{000}{001}, \stc{000}{010}, \stc{000}{011}, \stc{100}{000}, \stc{100}{001}, \stc{100}{010}, \stc{100}{011}\right\rbrace,$$
$$F_3=\left\lbrace\stc{000}{000}, \stc{000}{001}, \stc{000}{100}, \stc{000}{101}, \stc{010}{000}, \stc{010}{001}, \stc{010}{100}, \stc{010}{101}\right\rbrace,$$
$$F_4=\left\lbrace\stc{000}{000}, \stc{000}{001}, \stc{000}{110}, \stc{000}{111}, \stc{110}{000}, \stc{110}{001}, \stc{110}{110}, \stc{110}{111}\right\rbrace,$$
$$F_5=\left\lbrace\stc{000}{000}, \stc{000}{001}, \stc{010}{000}, \stc{010}{001}, \stc{100}{000}, \stc{100}{001}, \stc{110}{000}, \stc{110}{001}\right\rbrace,$$
$$F_6=\left\lbrace\stc{000}{000}, \stc{000}{010}, \stc{000}{100}, \stc{000}{110}, \stc{001}{000}, \stc{001}{010}, \stc{001}{100}, \stc{001}{110}\right\rbrace,$$
$$F_7=\left\lbrace\stc{000}{000}, \stc{000}{010}, \stc{000}{101}, \stc{000}{111}, \stc{101}{000}, \stc{101}{010}, \stc{101}{101}, \stc{101}{111}\right\rbrace,$$
$$F_8=\left\lbrace\stc{000}{000}, \stc{000}{010}, \stc{001}{000}, \stc{001}{010}, \stc{100}{000}, \stc{100}{010}, \stc{101}{000}, \stc{101}{010}\right\rbrace,$$
$$F_9=\left\lbrace\stc{000}{000}, \stc{000}{011}, \stc{000}{100}, \stc{000}{111}, \stc{011}{000}, \stc{011}{011}, \stc{011}{100}, \stc{011}{111}\right\rbrace,$$
$$F_{10}=\left\lbrace\stc{000}{000}, \stc{000}{011}, \stc{000}{101}, \stc{000}{110}, \stc{111}{000}, \stc{111}{011}, \stc{111}{101}, \stc{111}{110}\right\rbrace,$$
$$F_{11}=\left\lbrace\stc{000}{000}, \stc{000}{011}, \stc{011}{000}, \stc{011}{011}, \stc{100}{000}, \stc{100}{011}, \stc{111}{000}, \stc{111}{011}\right\rbrace,$$
$$F_{12}=\left\lbrace\stc{000}{000}, \stc{000}{100}, \stc{001}{000}, \stc{001}{100}, \stc{010}{000}, \stc{010}{100}, \stc{011}{000}, \stc{011}{100}\right\rbrace,$$
$$F_{13}=\left\lbrace\stc{000}{000}, \stc{000}{101}, \stc{010}{000}, \stc{010}{101}, \stc{101}{000}, \stc{101}{101}, \stc{111}{000}, \stc{111}{101}\right\rbrace,$$
$$F_{14}=\left\lbrace\stc{000}{000}, \stc{000}{110}, \stc{001}{000}, \stc{001}{110}, \stc{110}{000}, \stc{110}{110}, \stc{111}{000}, \stc{111}{110}\right\rbrace,$$
$$F_{15}=\left\lbrace\stc{000}{000}, \stc{000}{111}, \stc{011}{000}, \stc{011}{111}, \stc{101}{000}, \stc{101}{111}, \stc{110}{000}, \stc{110}{111}\right\rbrace.$$
We note that a different labeling of the Fano configurations is chosen in \cite{sturmfels}, where they are explicitly labeled by permutations of  the 7 elements of an Aronhold set. The choice of the above 15 Fano configurations giving a basis of
$W$ is of course not canonical, but just one possible choice. Note also that our Fano configurations are unordered and we don't have signs associated to them --- the
signs in our computations come from the theta transformation formulas and Riemann's quartic addition formula.

We now list the 15 invariant quartics in theta functions of the second order as given by Coble \cite{coblebook}, and also in \cite{sturmfels}: they are
$$Q_{000}:=Q_1=\sum_{\ep\in\FF_2^3} x_{\ep}^4;\quad Q_{\alpha}:=\frac12 \sum_{\ep\in\FF_2^3} x_{\ep}^2 x_{\ep+\alpha}^2;\quad Q_{\alpha'}:=\frac14 \sum_{\ep\in\FF_2^3}\prod_{\mu\in\alpha^\perp} x_{\ep+\mu},$$
%Q_{000}=Q_2:=x_{000}^4 +x_{001}^4 +x_{010}^4 +x_{100}^4 +x_{110}^4 +x_{101}^4 +x_{011}^4 +x_{111}^4,$$
%$$Q_{001}=Q_2:=x_{000}^2 x_{001}^2+x_{010}^2 x_{011}^2+x_{100}^2 x_{101}^2+x_{110}^2 x_{111}^2,$$
%$$Q_{010}=Q_3:=x_{000}^2 x_{010}^2+x_{001}^2 x_{011}^2+x_{100}^2 x_{110}^2+x_{101}^2 x_{111}^2,$$
%$$Q_{011}=Q_4:=x_{000}^2 x_{011}^2+x_{010}^2 x_{001}^2+x_{100}^2 x_{111}^2+x_{110}^2 x_{101}^2,$$
%$$Q_{100}=Q_5:=x_{000}^2 x_{100}^2+x_{010}^2 x_{110}^2+x_{001}^2 x_{101}^2+x_{011}^2 x_{111}^2,$$
%$$Q_{101}=Q_6:=x_{000}^2 x_{101}^2+x_{010}^2 x_{111}^2+x_{100 }^2 x_{001}^2+x_{110}^2 x_{011}^2,$$
%$$Q_{110}=Q_7:=x_{000}^2 x_{110}^2+x_{010}^2 x_{100}^2+x_{101}^2 x_{011}^2+x_{001}^2 x_{111}^2,$$
%$$Q_{111}=Q_8:=x_{000}^2 x_{111}^2+x_{010}^2 x_{101}^2+x_{100}^2 x_{011}^2+x_{110}^2 x_{001}^2,$$
%$$Q_{001'}=Q_9:=x_{000} x_{010}x_{100} x_{110}+x_{001} x_{011}x_{101} x_{111},$$
%$$Q_{010'}=Q_{10}:=x_{000} x_{001}x_{100}x_{101}+x_{010} x_{011}x_{110} x_{111},$$
%$$Q_{011'}=Q_{11}:=x_{000} x_{011}x_{100 }x_{111}+x_{001} x_{010}x_{101} x_{110},$$
%$$Q_{100'}=Q_{12}:=x_{000} x_{001}x_{010} x_{011}+x_{100} x_{101} x_{110} x_{111},$$
%$$Q_{101'}=Q_{13}:=x_{000} x_{010}x_{101 }x_{111}+x_{001} x_{011}x_{100} x_{110},$$
%$$Q_{110'}=Q_{14}:=x_{000} x_{001}x_{110}x_{111}+x_{010} x_{011}x_{100} x_{101},$$
%$$Q_{111'}=Q_{15}:=x_{000} x_{011}x_{101 }x_{110}+x_{001} x_{010}x_{100} x_{111},$$
for any $\alpha\in\FF_2^3\setminus\lbrace 0\rbrace$, where we recall (\ref{x}) that $x_\ep$ denotes $\Theta[\ep](\tau,z)$, and by $\alpha^\perp$ we mean the two dimensional subspace of $\FF_2^3$ orthogonal to $\alpha$ with respect to the usual scalar product, so that for example $011^\perp=\lbrace 000,011,100,111\rbrace$. The coefficients $1/2$ and $1/4$ are designed to make each monomial appear with coefficient one. Though the following formula would be simpler without them, we keep the notation above for easy comparision with \cite{coblebook} and \cite{sturmfels}.

The result of performing the computations described above, expressing all $H(F)$ in terms of the basis given, and
using the theta transformation formula is then the following
\begin{thm}\label{explicitcoble}
The Coble quartic can be written as a polynomial with 134 monomials as
%\noindent\begin{tabular}{rrrrrrrr}
% &\hskip-6mm$C(\tau,z)$&$\!\!\!=\ \ %s_{1}Q_{1}$&$-2s_{1}Q_{2}$&$-4s_{6}Q_{2}$&$-2s_{1}Q_{3}$&$-4s_{3}Q_{3}$&$-2s_{1}Q_{4}$\\
% &&$-4s_{9}Q_{4}$&$-2s_{1}Q_{5}$&$-4s_{2}Q_{5}$&$-2s_{1}Q_{6}$&$-4s_{7}Q_{6}$&$-2s_{1}Q_{7}$\\
% &&$-4s_{4}Q_{7}$&$-2s_{1}Q_{8}$&$+4s_{10}Q_{8}$&$+4s_{1}Q_{9}$&$+4s_{2}Q_{9}$&$+4s_{3}Q_{9}$\\
% &&$+4s_{4}Q_{9}$&$+8s_{5}Q_{9}$&$+4s_{1}Q_{10}$&$+4s_{2}Q_{10}$&$+4s_{6}Q_{10}$&$+4s_{7}Q_{10}$\\
% &&$+8s_{8}Q_{10}$&$+4s_{1}Q_{11}$&$+4s_{2}Q_{11}$&$+4s_{9}Q_{11}$&$-4s_{10}Q_{11}$&\!$+8s_{11}Q_{11}$\\
% &&$+4s_{1}Q_{12}$&$+4s_{3}Q_{12}$&$+4s_{6}Q_{12}$&$+4s_{9}Q_{12}$&$+8s_{12}Q_{12}$&$+4s_{1}Q_{13}$\\
% &&$+4s_{3}Q_{13}$&$+4s_{7}Q_{13}$&\!$-4s_{10}Q_{13}$&\!$+8s_{13}Q_{13}$&$+4s_{1}Q_{14}$&$+4s_{4}Q_{14}$\\
% &&$+4s_{6}Q_{14}$&$-4s_{10}Q_{14}$&$+8s_{14}Q_{14}$&$+4s_{1}Q_{15}$&$+4s_{4}Q_{15}$&$+4s_{7}Q_{15}$\\
% &&$+4s_{9}Q_{15}$&$+8s_{15}Q_{15}$,\!\!&&&&
%\end{tabular}
$$
\begin{aligned}
\! C(&\tau,z)= s_{1}Q_{000}-2(s_{1}+2s_{6})Q_{001}-2(s_{1}+2s_{3})Q_{010}-2(s_{1}+2s_{9})Q_{011}\\
&-2(s_{1}+2s_{2})Q_{100}-2(s_{1}+2s_{7})Q_{101}-2(s_{1}+2s_{4})Q_{110}\\ &-2(s_{1}
-2s_{10})Q_{111}
+8(s_{1}+s_{2}+s_{3}+s_{4}+2s_{5})Q_{001}'\\ &+8(s_{1}+s_{2}+s_{6}+s_{7}+2s_{8})Q_{010}'
+8(s_{1}+s_{2}+s_{9}-s_{10}+2s_{11})Q_{011}'\\ &+8(s_{1}+s_{3}+s_{6}+s_{9}+2s_{12})Q_{100}'
+8(s_{1}+s_{3}+s_{7}-s_{10}+2s_{13})Q_{101}'\\ &+8(s_{1}+s_{4}+s_{6}-s_{10}+2s_{14})Q_{110}'
+8(s_{1}+s_{4}+s_{7}+s_{9}+2s_{15})Q_{111}',
\end{aligned}
$$
where we denote $s_i:=H(F_i)$ for the basis of $W$ given above.
\end{thm}
\begin{rem}
Note that our expression for the Coble quartic is in terms of Fano configurations
of characteristics, and as such it has few terms. As explained in \cite{sturmfels},
each $s_i$ is in fact a degree 28 polynomial in theta constants of the second order, with a
huge number of monomials. A completely explicit formula for the Coble quartic, as a
polynomial in theta constants of the second order and $Q_i$'s, with 372060 monomials
altogether, is obtained in \cite[Sec.~7]{sturmfels}. In their notation we have
$r=H(F_1)$, $s_{\s}=-2H(F_1)\pm 4H(F_{\s})$, and a similar expression can be obtained for the $t_{\s}$.
Moreover, note that in the above formula if we take $-s_{10}$ instead of $s_{10}$, then all the signs becomes pluses.

Note that formula (\ref{kumsurexpl}) for the universal Kummer surface and the formula above for the Coble quartic, in terms of G\"opel systems, are remarkably similar, and it is of course natural to wonder whether there may be a similar interesting hypersurface in higher genus.
We hope that our formulas may lead to a better further understanding of the universal Kummer surface and of the Coble quartic.
\end{rem}
All of the computations described above were performed using Maple\textsuperscript{\textregistered},
and take a few minutes for a straightforward  unoptimized code on a regular PC.

\end{document}